\documentclass[12pt]{amsart}

\usepackage[all,cmtip]{xy}
\usepackage{graphicx}
\usepackage{amsmath}
\usepackage{amsfonts}
\usepackage{amssymb}
\usepackage{amsthm}
\usepackage{amscd}
\usepackage{textcomp}
\usepackage{hyperref}
\usepackage{extarrows}

\usepackage{scalerel}

\newcommand\reallywidehat[1]{\arraycolsep=0pt\relax%
\begin{array}{c}
\stretchto{
  \scaleto{
    \scalerel*[\widthof{\ensuremath{#1}}]{\kern-.5pt\bigwedge\kern-.5pt}
    {\rule[-\textheight/2]{1ex}{\textheight}} 
  }{\textheight} %
}{0.5ex}\\           
#1\\                 
\rule{-1ex}{0ex}
\end{array}
}

\newcommand{\cL}{\mathcal{L}}
\newcommand{\Adual}{A^{\vee}}

\DeclareFontEncoding{OT2}{}{} 

\newcommand{\hra}{\hookrightarrow}

\newcommand{\ssstyle}{\scriptscriptstyle}

\newcommand{\Q}{\mathbb{Q}}

\newcommand{\Z}{\mathbb{Z}}

\newcommand{\isom}{\cong}

\DeclareMathOperator{\coker}{coker}

\DeclareMathOperator{\Gal}{Gal}
\DeclareMathOperator{\im}{im}

\DeclareMathOperator{\Nm}{Nm}

\DeclareMathOperator{\Res}{Res}
\DeclareMathOperator{\Sel}{Sel}

\DeclareMathOperator{\res}{res}

\DeclareMathOperator{\ur}{nr}

\DeclareMathOperator{\Hom}{Hom}
\DeclareMathOperator{\TT}{TT}

\theoremstyle{plain}
\newtheorem{theorem}{Theorem}[section]
\newtheorem{proposition}[theorem]{Proposition}
\newtheorem{corollary}[theorem]{Corollary}

\newtheorem{lemma}[theorem]{Lemma}

\theoremstyle{definition}

\theoremstyle{remark}

\newtheorem{remark}[theorem]{Remark}

\numberwithin{equation}{section}

 1

\newcommand{\thetitle}
{Duality of Selmer groups of an abelian variety over a number field}

\makeatletter
      \def\@setcopyright{}
      \def\serieslogo@{}
      \makeatother

\begin{document}

\title{\thetitle} 
\author{Saikat Biswas}
\address{Department of Mathematical Sciences, The University of Texas at Dallas, Richardson, TX 75080-3021.}
\email{saikat.biswas@utdallas.edu}

\begin{abstract}
Let $A$ be an abelian variety defined over a number field $K$ and let $\Adual$ be the dual abelian variety. For an odd prime $p$, we consider two Selmer groups attached to $A[p]$ and relate the orders of these groups along with those of their corresponding duals to the order of the component groups of $\Adual$ at primes $v$.
\end{abstract}



\maketitle


\section{Introduction}
In order to motivate the main result of this paper, we begin with a quick review of \emph{generalized Selmer groups}, fixing notations along the way. Let $K$ be a number field with Galois group $G_K:=\Gal(\overline{K}/K)$, for a separable closure $\overline{K}$ of $K$. For every prime $v$ of $K$, let $K_v$ be the completion of $K$ at $v$ with Galois group $G_v:=G_{K_v}$. When $v$ is finite, we denote by $K_v^{\ur}$ the maximal unramified extension of $K_v$ whose Galois group is identified with the \emph{inertia group} $I_v\subset G_v$ at $v$. The  Galois group of the finite residue field $k_v$ of $K_v$ is given by the quotient group $G_v/I_v$. Let $M$ be a discrete $G_K$-module of finite cardinality, $\mu_{\infty}$ be the $G_K$-module of roots of unity in $\overline{K}$, and $M^*:=\Hom_{\Z}(M,\mu_{\infty})$ the \emph{Cartier dual} of $M$. We shall also denote by $\widehat{M}:=\Hom_{\Z}(M,\Q/\Z)$ the \emph{Pontryagin dual} of $M$. There is a continuous discrete action of $G_K$ on $M^*$ arising out of the corresponding action of $G_K$ on $M$ as well as on ${\overline{K}}^{\times}$. We consider a \emph{Selmer structure} $\cL$ for $M$, i.e. a family of subgroups $H^1_{\cL}(K_v,M)\subset H^1(K_v,M)$ for every prime $v$ such that $H^1_{\cL}(K_v,M)=H^1(G_v/I_v,M^{I_v})$ for all but finitely many $v$. The corresponding Selmer structure $\cL^*$ for $M^*$ is the family of subgroups $H^1_{\cL^*}(K_v,M^*)\subset H^1(K_v,M^*)$, where $H^1_{\cL^*}(K_v,M^*)$ is the orthogonal complement of $H^1_{\cL}(K_v,M)$ under the local Tate pairing 
$$H^1(K_v,M)\times H^1(K_v,M^*)\to\Q/\Z$$ 
The \emph{generalized Selmer group} $H^1_{\cL}(K,M)$ of $M$ is the subgroup of all classes in $H^1(K,M)$ whose image under the restriction map 
$$\Res_v:H^1(K,M)\to H^1(K_v,M)$$ 
lies in $H^1_{\cL}(K_v,M)$ for every $v$. The Selmer group $H^1_{\cL^*}(K,M^*)$ of $M^*$ is defined analogously, and is called the \emph{dual Selmer group} attached to $H^1_{\cL}(K,M)$. We have the following important result due to Wiles \cite{wiles}, suggested by a result of Greenberg \cite{gre}, which allows for the comparison of the order of a Selmer group to that of its dual.

\begin{theorem}\label{t1}
The groups $H^1_{\cL}(K,M)$ and $H^1_{\cL^*}(K,M^*)$ are finite, and we have
$$\frac{\#H^1_{\cL}(K,M)}{\#H^1_{\cL^*}(K,M^*)}=\frac{\#H^0(K,M)}{\#H^0(K,M^*)}\displaystyle\prod_{v}\frac{\#H^1_{\cL}(K_v,M)}{\#H^0(K_v,M)}$$
\end{theorem}

Since $\#H^1(G_v/I_v,M^{I_v})=\#H^0(K_v,M)$ \cite[\S2.3]{ddt}, all but finitely many terms in the product on the right hand side are $1$ so that the product makes sense. The formulation of Theorem \ref{t1} as given above appears in \cite{ddt} which is slightly different from that presented in \cite{wiles}, but the two are equivalent. Following a terminology suggested by Darmon in \cite{dar}, we refer to the term on the right side of Theorem \ref{t1} as the \emph{Euler characteristic} attached to $H^1_{\cL}(K,M)$ and denote it by $\chi_{\cL}(K,M)$. Thus, we can write

\begin{equation}\label{e1}
\chi_{\cL}(K,M)=\frac{\#H^0(K,M)}{\#H^0(K,M^*)}\displaystyle\prod_{v}\frac{\#H^1_{\cL}(K_v,M)}{\#H^0(K_v,M)}
\end{equation}

If $\cL_1,\cL_2$ are Selmer structures for $M$, we say that $\cL_1\subseteq\cL_2$ if $H^1_{\cL_1}(K_v,M)\subseteq H^1_{\cL_2}(K_v,M)$ for every $v$. In this case, it can be shown that $H^1_{\cL_1}(K,M)\subseteq H^1_{\cL_2}(K,M)$ and that
$\cL_1^*\supseteq\cL_2^*$. The next corollary follows immediately from (\ref{e1}) above.

\begin{corollary}\label{c1}
Suppose that $\cL_1\subseteq\cL_2$ are Selmer structures for $M$. Then we have
\begin{equation}\label{e2}
\frac{\chi_{\cL_2}(K,M)}{\chi_{\cL_1}(K,M)}=\displaystyle\prod_{v}\frac{\#H^1_{\cL_2}(K_v,M)}{\#H^1_{\cL_1}(K_v,M)}
\end{equation}
\end{corollary}

In particular, the comparison of the Euler characteristics attached to two different Selmer groups is a purely local calculation. We can also rewrite (\ref{e2}), in view of Theorem \ref{t1}, as

\begin{equation}\label{e3}
\frac{\#H^1_{\cL_2}(K,M)}{\#H^1_{\cL_2^*}(K,M^*)}=\chi_{\cL_1}(K,M)\cdot\displaystyle\prod_{v}\frac{\#H^1_{\cL_2}(K_v,M)}{\#H^1_{\cL_1}(K_v,M)}
\end{equation}

The goal of this paper is to carry out the computation involved in Corollary \ref{c1} for the case when, for an odd prime $p$, $M=A[p]$ is the $p$-torsion subgroup of an abelian variety $A$ defined over $K$. Let $\Adual$ be the dual abelian variety. Then it can be shown that $M^*=\Adual[p]$. Let $S$ be any finite set of primes of $K$ containing all archimedean primes, all primes above $p$, and all primes of bad reduction for $A$. We consider three Selmer structures for $A[p]$. The first, denoted by $\Sel_p(A/K_v)$, is defined by $\im(\delta_v)$, where $\delta_v$ is the injection 
$$A(K_v)/pA(K_v)\hra H^1(K_v,A[p])$$ 
arising out of the \emph{Kummer exact sequence}. The second, denoted by $H^1_{(S)}(K_v,A[p])$, is defined by $\ker(\varphi)$, where $\varphi$ is the composition 
$$H^1(K_v,A[p])\to H^1(K_v,A)[p]\to H^1(K_v^{\ur},A)[p]$$
The third, denoted by $H^1_{[S]}(K_v,A[p])$, is defined by $\delta_v(\kappa_v)$, where $\kappa_v$ is the kernel of the map
$$A(K_v)/pA(K_v)\to\Phi_{A,v}(k_v)/p\Phi_{A,v}(k_v)$$
Here, $\Phi_{A,v}(k_v)$ is the \emph{arithmetic component group} of $A$ at $v$.

The respective Selmer groups, all subgroups of $H^1(K,A[p])$, are denoted by $\Sel_p(A/K),\;H^1_{(S)}(K,A[p])$, and $H^1_{[S]}(K,A[p])$. The corresponding Selmer groups for $\Adual[p]$ are defined similarly. Our main result in this paper is the following theorem.

\begin{theorem}[Main Theorem]\label{mt}
Under the notations introduced above, we have
$$\frac{\chi_{(S)}(K,A[p])}{\chi_{\ssstyle{\Sel_p}}(K,A[p])}=\#\bigoplus_{v\in{S}}\Phi_{\Adual,v}(k_v)[p]$$
or, equivalently
$$\frac{\#H^1_{(S)}(K,A[p])}{\#H^1_{[S]}(K,\Adual[p])}=\frac{\#\Sel_p(A/K)}{\#\Sel_p(\Adual/K)}\cdot\#\bigoplus_{v\in{S}}\Phi_{\Adual,v}(k_v)[p] $$
\end{theorem}

The local terms on the right hand side of Theorem \ref{mt} can be computed explicitly by Tate's algorithm \cite[IV.\S9]{silverman} when $A$ is an elliptic curve, by the monodromy pairing \cite[\S10]{gro} when $A$ is a semistable abelian variety, and by the intersection pairing \cite[\S1]{bl} when $A$ is the Jacobian of a curve.


\begin{corollary}\label{c2}
Suppose that $A$ is an elliptic curve defined over $K$ and $p$ is a prime of good reduction for $A$. Then
$$\frac{\#H^1_{(S)}(K,A[p])}{\#H^1_{[S]}(K,A[p])}=\#\bigoplus_{v\in{S}}\Phi_{A,v}(k_v)[p]$$
and
$$\#\Sel_p(A/K)<\#H^1_{[S]}(K,A[p])\;\#\bigoplus_{v\in{S}}\Phi_{A,v}(k_v)[p]$$
\end{corollary}

\begin{proof}
In this case, $\Adual\isom A$ and $\chi_{\ssstyle{\Sel_p}}(K,A[p])=1$ \cite[Prop. 10.15]{dar}. The first equality then follows immediately from Theorem \ref{mt}. To prove the second inequality, we show in Lemma \ref{l4} below that there are inclusions
$$H^1_{[S]}(K,A[p])\subset\Sel_p(A/K)\subset H^1_{(S)}(K,A[p])$$
which in turn imply that
$$\#H^1_{[S]}(K,A[p])<\#\Sel_p(A/K)<\#H^1_{(S)}(K,A[p])$$
The desired inequality now follows by dividing throughout by $\#H^1_{[S]}(K,A[p])$ and using the first equality.
\end{proof}
In Section 2, we investigate the different Selmer structures on $A[p]$ and the corresponding dual Selmer structures on $\Adual[p]$. We relate the order of these structures to each other using Tate local duality as our key ingredient. In Section 3, we consider the corresponding Selmer groups on $A[p]$ as well as on $\Adual[p]$ and, using Corollary \ref{c1}, prove our main result. In Section 3, we present an alternative proof of our main theorem by deriving it from a refinement of a well known dual exact sequence due to Cassels.

\section{Local Computations}
We collect some definitions and results from \cite{bis}. Let $\Phi_{A,v}$ be the $k_v$-group scheme of connected components of the \emph{Neron model} of $A$, and let $\Phi_{A,v}(k_v)$ the \emph{arithmetic component group} of $A$ at $v$.  Let 
$$c_{\ssstyle{A,v}}=\#\Phi_{A,v}(k_v)$$ 
be the Tamagawa number of $A$ at $v$. We write
$$c(A/K)=\prod_{v}c_{\ssstyle{A,v}}$$

Equivalently, we can also define $\Phi_{A,v}(k_v)$ by the exactness of the sequence
\begin{equation}\label{e4}
1\to A_0(K_v)\to A(K_v)\to \Phi_{A,v}(k_v)\to 1
\end{equation}
Here, $A_0(K_v)$ is the subgroup of \emph{universal norms} of $A(K_v)$ from $K_v^{\ur}$ defined by
$$A_0(K_v):=\Nm(A(K_v^{\ur}))=\bigcap_{L_i}N_{L_i/K_v}(A(L_i))$$
where, the intersection is over all finite, unramified extensions $L_i$ of $K_v$ and $N_{L_i/K_v}$ is the norm map from $L_i$ to $K_v$. The group $\TT(A/K_v)$ of \emph{Tamagawa torsors} of $A/K_v$ is defined by the exactness of the sequence
\begin{equation}\label{e5}
1\to \TT(A/K_v)\to H^1(K_v,A)\xrightarrow{\res_v} H^1(K_v^{\ur},A)
\end{equation}
In particular, $\TT(A/K_v)$ is trivial when $A$ has good reduction at $v$. There is a perfect duality of finite groups
\begin{equation}\label{e6}
\TT(A/K_v)\times\Phi_{\Adual,v}(k_v)\to\Q/\Z
\end{equation}
Let $p$ be an odd prime. For every prime $v$, the \emph{Kummer exact sequence} for $A/K_v$ is
\begin{equation}\label{e7}
1\to A(K_v)/pA(K_v)\xrightarrow{\delta_v} H^1(K_v,A[p])\xrightarrow{\pi_v} H^1(K_v,A)[p]\to 1
\end{equation}
Let 
$$\Sel_p(A/K_v):=\im(\delta_v)\subset H^1(K_v,A[p])$$
Since $\delta_v$ is injective, it follows immediately that $\#\Sel_p(A/K_v)=\#A(K_v)/pA(K_v)$. Furthermore, for $v\notin{S}$, $A$ has good reduction at $v$ and $\Sel_p(A/K_v)$ is the subgroup of all classes in $H^1(K_v,A[p])$ that are unramified at $v$, i.e. their restriction to $H^1(K_v^{\ur},A[p])$ is trivial \cite[Prop.C.4.6]{hs}. In other words, $\Sel_p(A/K_v)=H^1(K_v^{\ur}/K_v,A(K_v^{\ur})[p])$ for all $v\notin{S}$. We have thus proved
\begin{lemma}\label{l1}
The family of subgroups $\Sel_p(A/K_v)\subset H^1(K_v,A[p])$ for all $v$ form a Selmer structure for $A$. Furthermore, we have
$$\#\Sel_p(A/K_v)=\#A(K_v)/pA(K_v)$$
\end{lemma}
The Kummer exact sequence for $\Adual/K_v$ is
\begin{equation}\label{e8}
1\to \Adual(K_v)/p\Adual(K_v)\xrightarrow{\delta_v^{\vee}} H^1(K_v,\Adual[p])\xrightarrow{\pi_v^{\vee}} H^1(K_v,\Adual)[p]\to 1
\end{equation}
We note that $\Adual[p]$ is the Cartier dual of $A[p]$ \cite[\S 15]{mum}. This also follows from the non-degenerate Weil pairing
$$A[p]\times\Adual[p]\to\mu_p$$ 
There is an isomorphism \cite[Cor I.3.4]{adt} of compact groups
$$\widehat{A(K_v)}\isom H^1(K_v,\Adual)$$
and an isomorphism  \cite[Cor I.2.3]{adt} of finite groups
$$\widehat{H^1(K_v,A[p])}\isom H^1(K_v,\Adual[p])$$
that allow us to identify the exact sequence (\ref{e8}) as the Pontryagin dual of the exact sequence (\ref{e7}). In particular, 
$$\Sel_p(\Adual/K_v)=\im(\delta_v^{\vee})\subset H^1(K_v,\Adual[p])$$ 
is the orthogonal complement of 
$$\Sel_p(A/K_v)=\im(\delta_v)\subset H^1(K_v,A[p])$$
under the local Tate pairing
\begin{equation}\label{e9}
H^1(K_v,A[p])\times H^1(K_v,\Adual[p])\to\Q/\Z
\end{equation}
It follows from Lemma \ref{l1} that
\begin{lemma}\label{l2}
The family of subgroups $\Sel_p(\Adual/K_v)\subset H^1(K_v,\Adual[p])$ for all $v$ form a Selmer structure for $\Adual$. 
\end{lemma}

Let 
$$H^1_{(S)}(K_v,A[p]):=\ker\varphi_v$$ 
where $\varphi_v$ is the composition
\[
\xymatrix{
H^1(K_v,A[p])\ar[r]^{\pi_v}\ar[dr]^{\varphi_v}&H^1(K_v,A)[p]\ar[d]^{\res_v^p}\\ &H^1(K_v^{\ur},A)[p]}
\]

\begin{theorem}\label{t2}
There is an exact sequence
\begin{equation}\label{e10}
1\to A(K_v)/pA(K_v)\xrightarrow{\delta_v} H^1_{(S)}(K_v,A[p])\to \TT(A/K_v)[p]\to 1
\end{equation}
\end{theorem}

\begin{proof}
The pair of maps $\pi_v,\res_v^p$ yields the exact sequence
$$1\to\ker(\pi_v)\to\ker(\varphi_v)\to\ker(\res_v^p)\to\coker(\pi_v)$$
From (\ref{e7}), we find that $\ker(\pi_v)=\im(\delta_v)\isom A(K_v)/pA(K_v)$, and $\coker(\pi_v)=1$ since $\pi_v$ is surjective. From (\ref{e5}), we get that $\ker(\res_v^p)\isom\TT(A/K_v)[p]$. The desired exact sequence now follows.
\end{proof}

\begin{corollary}\label{c3}
The family of subgroups $H^1_{(S)}(K_v,A[p])\subset H^1(K_v,A[p])$ for all $v$ form a Selmer structure for $A$. Furthermore, we have
\begin{equation}\label{e11}
\frac{\#H^1_{(S)}(K_v,A[p])}{\#\Sel_p(A/K_v)}=\#\Phi_{\ssstyle{\Adual,v}}(k_v)[p]
\end{equation}
\end{corollary}

\begin{proof}
As noted above, $A$ has good reduction at all primes $v\notin{S}$ so that $\TT(A/K_v)$ is trivial at these primes. Thus, we have
$$H^1_{(S)}(K_v,A[p])=\ker(\varphi_v)\isom\im(\delta_v)=\Sel_p(A/K_v)$$
which shows that the subgroups $H^1_{(S)}(K_v,A[p])$ form a Selmer structure for $A$. Eq.(\ref{e11}) follows directly from Eq.(\ref{e10}) and Eq.(\ref{e6}), noting that $\Phi_{\ssstyle{\Adual,v}}(k_v)$ is finite.
\end{proof}

\begin{remark}
We note that 
$$\#\Phi_{\ssstyle{\Adual,v}}(k_v)=c_{\ssstyle{\Adual,v}}=c_{\ssstyle{A,v}}=\#\Phi_{\ssstyle{A,v}}(k_v)$$ 
by \cite[Prop 4.3]{lorenzini}. It follows that
$$c(A/K)=\prod_{v}c_{\ssstyle{A,v}}=\prod_{v}c_{\ssstyle{\Adual,v}}=c(\Adual/K)$$
\end{remark}

Let $\kappa_v$ be defined by the exactness of the sequence
\begin{equation}\label{e12}
1\to\kappa_v\to A(K_v)/pA(K_v)\xrightarrow{\alpha_v}\Phi_{A,v}(k_v)/p\Phi_{A,v}(k_v)\to 1
\end{equation}
Let 
$$H^1_{[S]}(K_v,A[p]):=\delta_v(\kappa_v)\subset H^1(K_v,A[p])$$ 

The following lemma is evident.

\begin{lemma}
We have
\begin{equation}
\#H^1_{[S]}(K_v,A[p])=\#\kappa_v=\#[A(K_v)\cap{p}\Phi_{\ssstyle{A,v}}(k_v)]/pA(K_v)
\end{equation}
\end{lemma}

\begin{theorem}\label{t3}
$H^1_{[S]}(K_v,\Adual[p])$ and $H^1_{(S)}(K_v,A[p])$ are orthogonal complements of each other under the local Tate pairing (\ref{e9}).
\end{theorem}

\begin{proof}
Consider the commutative diagram
\[
\xymatrix{
1\ar[r] &\Adual(K_v)/p\Adual(K_v)\ar[r]^{\delta_v^{\vee}}\ar[d]^{\alpha_v^{\vee}} &H^1(K_v,\Adual[p])\ar[r]^{\pi_v^{\vee}}\ar[d]^{\beta_v} &H^1(K_v,\Adual)[p]\ar[r]\ar[d]^{||} &1\\
1\ar[r] &\Phi_{\Adual,v}(k_v)/p\Phi_{\Adual,v}(k_v)\ar[r] &\widehat{H^1_{(S)}(K_v,A[p])}\ar[r]^{\widehat{\delta_v}} &H^1(K_v,\Adual)[p]\ar[r] &1}
\]
The top row is the exact sequence (\ref{e8}), while the bottom row is the dual of the exact sequence (\ref{e10}). The right vertical map is an identity, while the left vertical map $\alpha_v^{\vee}$ follows from (\ref{e12}), whence it is surjective. It follows that the middle vertical map $\beta_v$, induced by the two boundary maps, is surjective and that
$$\ker(\beta_v)\isom\delta_v^{\vee}(\ker(\alpha_v^{\vee}))=H^1 _{[S]}(K_v,\Adual[p])$$
\end{proof}

\begin{corollary}\label{c4}
The family of subgroups $H^1_{[S]}(K_v,\Adual[p])\subset H^1(K_v,\Adual[p])$ for all $v$ form a Selmer structure for $\Adual$. Furthermore,
$$\frac{\#\Sel_p(\Adual/K_v)}{\#H^1_{[S]}(K_v,\Adual[p])}=\#\Phi_{\ssstyle{\Adual,v}}(k_v)[p]$$
\end{corollary}

Lastly, we take note of the following lemma whose proof follows directly from definition.

\begin{lemma}\label{l3}
The aforementioned Selmer structures have inclusions
$$H^1_{[S]}(K_v,A[p])\subset\Sel_p(A/K_v)\subset H^1_{(S)}(K_v,A[p])$$
for all $v\in{S}$.
\end{lemma}

\section{Global Computations}
Let $K_S$ be the maximal extension of $K$ unramified outside $S$ and let $G_{S}=\Gal(K_S/K)$. For each $v\in{S}$, let $w$ be some prime of $K_S$ above $v$, and let $K_{S,w}$ be the union of the completions at $w$ of the finite extensions of $K$ contained in $K_S$. The composition
$$G_{K_v}\twoheadrightarrow\Gal(K_{S,w}/K_v)\hra G_S$$
induces the restriction map
\begin{equation}
\Res_v\;:\;H^1(G_S,A[p])\longrightarrow H^1(K_v,A[p])
\end{equation}
which in turn give the homomorphism
\begin{equation}
\beta_{S}\;:\;H^1(G_S,A[p])\longrightarrow\bigoplus_{v\in{S}} H^1(K_v,A[p])
\end{equation}
For each of the Selmer structures introduced in the previous section, we define the corresponding Selmer groups by 
\begin{align}
\Sel_p(A/K)&:=\beta_{S}^{-1}\left(\bigoplus_{v\in{S}}\Sel_p(A/K_v)\right)\subset H^1(G_S,A[p])\\
H^1_{(S)}(K,A[p])&:=\beta_{S}^{-1}\left(\bigoplus_{v\in{S}}H^1_{(S)}(K_v,A[p])\right)\subset H^1(G_S,A[p])\\
H^1_{[S]}(K,A[p])&:=\beta_{S}^{-1}\left(\bigoplus_{v\in{S}}H^1_{[S]}(K_v,A[p])\right)\subset H^1(G_S,A[p])
\end{align}
Since $H^1(G_S,A[p])$ is finite \cite[Cor. I.4.15]{adt}, it follows that each of the Selmer groups above are also finite. It also follows from Lemma \ref{l3} above that 

\begin{lemma}\label{l4}
There are proper inclusions
$$H^1_{[S]}(K,A[p])\subset\Sel_p(A/K)\subset H^1_{(S)}(K,A[p])$$
\end{lemma}

The following proposition will be used in the next section.

\begin{proposition}\label{p1}
There are exact sequences
\begin{align*}
1\to\Sel_p(A/K)\to H^1(G_S,A[p])&\xrightarrow{\lambda_S}\bigoplus_{v\in{S}}H^1(K_v,A)[p]\\
1\to\Sel_p(A/K)\to H^1_{(S)}(K,A[p])&\xrightarrow{\rho_S}\bigoplus_{v\in{S}}\TT(A/K_v)[p]\\
1\to H^1_{[S]}(K,A[p])\to\Sel_p(A/K)&\xrightarrow{\eta_S}\bigoplus_{v\in{S}}\Phi_{\ssstyle{A,v}}(k_v)/p\Phi_{\ssstyle{A,v}}(k_v)
\end{align*}
\end{proposition}

\begin{proof}
Lemmas \ref{l3} and \ref{l4} imply the exact sequences
\begin{align*}
1\to\Sel_p(A/K)\to H^1(G_S,A[p])&\to\bigoplus_{v\in{S}}H^1(K_v,A[p])/\im(\delta_v)\\
1\to\Sel_p(A/K)\to H^1_{(S)}(K,A[p])&\to\bigoplus_{v\in{S}}H^1_{(S)}(K_v,A[p])/\im(\delta_v)\\
1\to H^1_{[S]}(K,A[p])\to\Sel_p(A/K)&\to\bigoplus_{v\in{S}}\im(\delta_v)/\delta_v(\kappa_v)
\end{align*}
The proposition now follows from Eqs (\ref{e7}), (\ref{e10}), and (\ref{e12}).
\end{proof}

Furthermore, we find that the dual Selmer groups attached to each of these aforementioned Selmer groups can be identified as
\begin{align}
\Sel_{p^*}(\Adual/K)&=\Sel_p(\Adual/K)\\
H^1_{(S)^*}(K,\Adual[p])&=H^1_{[S]}(K,\Adual[p])
\end{align}
The Euler characteristics attached to the Selmer groups $\Sel_p(A/K)$ and $H^1_{(S)}(K,A[p])$ can be given, using Theorem \ref{t1}, as
\begin{align}
\chi_{\Sel_p}(K,A[p])&=\frac{\#\Sel_p(A/K)}{\#\Sel_p(\Adual/K)}\\
\chi_{(S)}(K,A[p])&=\frac{\#H^1_{(S)}(K,A[p])}{\#H^1_{[S]}(K,\Adual[p])}
\end{align}
Using Eqs (\ref{e2}) and (\ref{e11}), we obtain
$$\frac{\chi_{(S)}(K,A[p])}{\chi_{\Sel_p}(K,A[p])}=\prod_{v}\frac{\#H^1_{(S)}(K_v,A[p])}{\#\Sel_p(A/K_v)}=\#\bigoplus_{v\in{S}}\Phi_{\Adual,v}(k_v)[p]$$
which proves Theorem \ref{mt}. 

\begin{remark}
$\Sel_p(A/K)$ is usually referred to as the classical $p$-Selmer group of $A$. On the other hand, we may refer to $H^1_{(S)}(K,A[p])$ as the \emph{relaxed} Selmer group of $A$ at $S$ and to $H^1_{[S]}(K,A[p])$ as the \emph{restricted} Selmer group of $A$ at $S$. We note that these notations and terminologies (along with the construction of $S$) have been used differently in \cite{dar}.
\end{remark}

\section{A Dual Exact Sequence}

We present an alternative proof of Theorem \ref{mt} in this section. The following result is often referred to as the Cassels-Poitou-Tate sequence. The original version of this theorem appeared in \cite{cas} which involved elliptic curves.

\begin{theorem}\label{cas}
There is an exact sequence
\[\xymatrix{
1\ar[r]&\Sel_p(A/K)\ar[r]&H^1(G_S,A[p])\ar[r]^{\lambda_{S}}&\displaystyle\bigoplus_{v\in{S}}H^1(K_v,A)[p]\ar[d]\\
&\displaystyle\bigoplus_{v\in{S}}H^2(K_v,A[p])\ar[d]&H^2(G_S,A[p])\ar[l]&\widehat{\Sel_p(\Adual/K)}\ar[l]\\
&\widehat{H^0(G_S,\Adual[p])}\ar[r]&1
}
\]
\end{theorem}

The results obtained in the last section can be used to establish the following refinement to Theorem \ref{cas}. Note that Theorem \ref{mt} is an immediate consequence of Theorem \ref{t4}.

\begin{theorem}\label{t4}
There is an exact sequence of finite groups
\[\xymatrix{
1\ar[r]&\Sel_p(A/K)\ar[r]&H^1_{(S)}(K,A[p])\ar[r]&\displaystyle\bigoplus_{v\in{S}}\TT(A/K_v)[p]\ar[d]\\
&1&\widehat{H^1_{[S]}(K,\Adual[p])}\ar[l]&\widehat{\Sel_p(\Adual/K)}\ar[l]
}
\]
\end{theorem}

\begin{proof}
Consider the following exact sequences from Proposition \ref{p1}
\begin{align*}
1\to\Sel_p(A/K)\to H^1_{(S)}(K,A[p])&\xrightarrow{\rho_S}\bigoplus_{v\in{S}}\TT(A/K_v)[p]\\
1\to H^1_{[S]}(K,\Adual[p])\to\Sel_p(\Adual/K)&\xrightarrow{\eta_S^{\vee}}\bigoplus_{v\in{S}}\Phi_{\ssstyle{\Adual,v}}(k_v)/p\Phi_{\ssstyle{\Adual,v}}(k_v)
\end{align*}
It suffices to prove that the images of $\rho_{S}$ and $\eta_{S}^{\vee}$ are orthogonal complements of each other under the sum of the local pairings (\ref{e6}). Equivalently, we prove that the sequence
$$H^1_{(S)}(K,A[p])\xrightarrow{\rho_{S}}\bigoplus_{v\in{S}}\TT(A/K_v)[p]\xrightarrow{\widehat{\eta_{S}^{\vee}}}\widehat{\Sel_p(\Adual/K)}$$
is exact at the middle term, i.e. $\im{\rho_{S}}=\ker{\widehat{\eta_{S}^{\vee}}}$. The commutative diagram
\[\xymatrixcolsep{3pc}\xymatrix{
H^1(G_S,A[p])\ar[d]^{\lambda_{S}}\ar[dr]^{\delta_{S}}&\\
\displaystyle\bigoplus_{v\in{S}}H^1(K_v,A)[p]\ar[r]^{\oplus\res_v^p}&\displaystyle\bigoplus_{v\in{S}}H^1(K_v^{\ur},A)[p]
}
\]
yields the exact sequence
$$1\to\ker{\lambda_{S}}\to\ker{\delta_{S}}\to\ker(\oplus\res_v^p)\xrightarrow{\phi}\coker{\lambda_{S}}\to\coker{\delta_{S}}$$
part of which we identify as
$$H^1_{(S)}(K,A[p])\xrightarrow{\rho_{S}}\bigoplus_{v\in{S}}\TT(A/K_v)[p]\xrightarrow{\phi}\coker{\lambda_{S}}$$
Hence, we have $\im{\rho_{S}}=\ker{\phi}$. On the other hand, Theorem \ref{cas} yields an inclusion
$$\coker{\lambda_{S}}\hra\widehat{\Sel_p(\Adual/K)}$$
The map $\widehat{\eta_{S}}$ is the composition
$$\bigoplus_{v\in{S}}\TT(A/K_v)[p]\xrightarrow{\phi}\coker{\lambda_{S}}\hra\widehat{\Sel_p(\Adual/K)}$$
which implies that $\ker{\phi}=\ker{\widehat{\eta_{S}}}$. Hence, $\im{\rho_{S}}=\ker{\widehat{\eta_{S}}}$.

\end{proof}

\begin{remark}
More generally, if $\cL_1\subseteq\cL_2$ are Selmer structures for $M$ then there are exact sequences
\[\xymatrix{
1\ar[r]&H^1_{\cL_1}(K,M)\ar[r]&H^1_{\cL_2}(K,M)\ar[r]&\displaystyle\bigoplus_{v}H^1_{\cL_2}(K_v,M)/H^1_{\cL_1}(K_v,M)\\
1\ar[r]&H^1_{\cL_2^*}(K,M^*)\ar[r]&H^1_{\cL_1^*}(K,M^*)\ar[r]&\displaystyle\bigoplus_{v}H^1_{\cL_1^*}(K_v,M^*)/H^1_{\cL_2^*}(K_v,M^*)
}
\]
such that the images of the two right-hand maps are orthogonal complements of each other under the local Tate pairing \cite[Thm. 1.7.3]{rubin}. Consequently, there is a canonical exact sequence of finite groups
\[\xymatrix{
1\ar[r]&H^1_{\cL_1}(K,M)\ar[r]&H^1_{\cL_2}(K,M)\ar[r]&\displaystyle\bigoplus_{v}H^1_{\cL_2}(K_v,M)/H^1_{\cL_1}(K_v,M)\ar[d]\\
&1&\widehat{H^1_{\cL_2^*}(K,M^*)}\ar[l]&\widehat{H^1_{\cL_1^*}(K,M^*)}\ar[l]
}
\]
\end{remark}
Note that Eq.(\ref{e3}) follows immediately from this sequence.

\newpage
\newcommand{\etalchar}[1]{$^{#1}$}

\end{document}